\documentclass[12pt,a4paper]{article}

\usepackage{amsthm,amsmath,amsfonts,amssymb,yfonts}
\usepackage[english]{babel}
\usepackage{graphicx}
\usepackage{verbatim}
\usepackage{bm}
\usepackage{color}
\usepackage{float}
 \usepackage{booktabs}
 \usepackage[table,xcdraw]{xcolor}

\usepackage{enumitem}

\usepackage{tikz,lipsum,lmodern}
\usepackage[most]{tcolorbox}
\usepackage{multicol}
\usepackage{tabularx}
\usepackage[utf8]{inputenc}
\usepackage{array}
\usepackage{wrapfig}
\usepackage{multirow}
\usepackage{lipsum}

\usepackage[utf8]{inputenc}
\usepackage[T1]{fontenc}
\usepackage{authblk}
\usepackage{hyperref}

\numberwithin{equation}{subsection}

\let\oldsection\section
\renewcommand{\section}{
	\renewcommand{\theequation}{\thesection.\arabic{equation}}
	\oldsection}
\let\oldsubsection\subsection
\renewcommand{\subsection}{
	\renewcommand{\theequation}{\thesubsection.\arabic{equation}}
	\oldsubsection}

\newtheorem{theorem}{Theorem}[section]
\newtheorem{proposition}[theorem]{Proposition}
\newtheorem{lemma}[theorem]{Lemma}

\newtheorem{corollary}[theorem]{Corollary}

\newtheorem{definition}[theorem]{Definition}

\newtheorem{remark}[theorem]{Remark}

\newtheorem{problem}[theorem]{Problem}

\newcommand{\G}{\Gamma}

\newcommand{\PP}{{\mathcal P}}
\newcommand{\BB}{{\mathcal B}}

\newcommand{\R}{{\cal R}}

\title{Characterizing bipartite distance-regularized graphs with vertices of eccentricity 4}

\author{
	{Blas Fern\'andez}\\
	{\small Andrej Marušič Institute}\\
	{\small University of Primorska}\\
	{\small Muzejski trg 2, 6000 Koper, Slovenia }\\
	{\small blas.fernandez@famnit.upr.si} \and 
	{Marija Maksimović}\\
	{\small Faculty of Mathematics}\\
	{\small University of Rijeka}\\
	{\small Radmile Matejčić 2, 51000 Rijeka, Croatia }\\
	{\small mmaksimovic@math.uniri.hr} \and 
	{Sanja Rukavina}\\
	{\small Faculty of Mathematics}\\
	{\small University of Rijeka}\\
	{\small Radmile Matejčić 2, 51000 Rijeka, Croatia }\\
	{\small sanjar@math.uniri.hr}
}
\begin{document}

\maketitle

\begin{abstract}
The characterization of bipartite distance-regularized graphs, where some vertices have eccentricity less than four, in terms of the incidence structures of which they are incidence graphs, is known. In this paper we prove that there is a one-to-one correspondence between the incidence graphs of  quasi-symmetric SPBIBDs with parameters $(v,b,r,k, \lambda_1,0)$ of type $(k-1,t)$ with intersection numbers $x=0$ and $y>0$, where $0< y\leq t<k$ , and bipartite distance-regularized graphs with $D=D'=4$.
\end{abstract}



\section{Introduction}
\label{sec:intro}

We assume familiarity with the basic facts and notions from graph theory and from the theory of combinatorial designs.  For background reading we refer the reader to \cite{BCN, colbourn2007crc, Die, SS-QS}. In this paper, $\Gamma=(X, \R)$ will denote a finite, undirected, connected graph, without loops and multiple edges, with vertex set $X$ and edge set $\R$. 
An incidence structure ${\mathcal D} =( {\mathcal P}, {\mathcal B}, I)$,
with point set ${\mathcal P}$, block set ${\mathcal B}$ and incidence relation $I \subseteq {\mathcal P} \times {\mathcal B}$, 
where $|{\mathcal P}|=v$, $|{\mathcal B}|=b$, each block $B \in {\mathcal B}$ is incident with exactly $k$ points, 
every $t$-tuple of distinct points from $ {\mathcal P}$ is incident with exactly $\lambda$ blocks and each point is incident with exactly $\displaystyle r$ blocks is a {\it $t$-$(v,b,r,k,\lambda)$ design} or a {\it $t$-$(v,k,\lambda)$ design}.  We will only consider  {\it $t$-$(v, k, \lambda)$} designs that are simple, proper and nontrivial, and to rule out degenerate cases, we will assume that the parameters of a design satisfy $v>k>t\geq 1$ and $\lambda\geq 1$.

 Consider a graph $\Gamma=(X, \R)$, and, for any $x, y\in X$, denote by $\partial(x, y)$ the {\it distance} between $x$ and $y$ (the length of a shortest walk from $x$ to $y$). The {\it diameter of $\G$} is defined to be 
$\max\{\partial(u,v)\,|\,u, v\in X\},$ and the {\it eccentricity} of $x$, denoted by $\varepsilon=\varepsilon(x)$, is the maximum distance between $x$ and any other vertex of $\G$. Note that the diameter of $\G$ equals $\max\{\varepsilon(x)\mid x\in X\}$. For an integer $i$ we define  
$
\G_i(x)=\left\lbrace y \in X \mid \partial(x, y)=i\right\rbrace. 
$
Notice that $\G_{i}(x)$ is empty if and only if $i<0$ or $i>\varepsilon(x)$, and $\G_1(x)$ is the set of neighbours of $x$. We will abbreviate $\G(x)=\G_1(x)$. We say that a vertex $x\in X$ has {\it valency} $k$ if $|\G(x)|=k$. A graph $\G$ is called {\it regular} if every vertex has the same valency, i.e., if there is a non-negative integer $k$ such that $|\G(x)|=k$ for every vertex $x\in X$. In this case we also say that $\G$ is regular with {\it valency} $k$ or $k$-regular. 

 A {\it bipartite} (or {\it $(Y,Y')$-bipartite}) graph is a graph whose vertex set can be partitioned into two subsets $Y$ and $Y'$ such that each edge has one end in $Y$ and one end in $Y'$. The vertex sets $Y$ and $Y'$ in such a partition {are} called  {\it color partitions} (or {\it bipartitions}) of the graph. A bipartite graph $\G$ with color partitions $Y$ and $Y'$ is said to be {\it biregular} if the valency of a vertex only
 depends on the color partition where it belongs to; see for instance \cite{DKT}. 
 
 For vertices $x_1,x_2,\hdots,x_k\in X$ and non-negative integers $i_1,i_2,\hdots,i_k$ $(0\le i_1,i_2,\hdots,i_k \le d)$ we define
$
\G_{i_1,i_2,\hdots,i_k}(x_1,x_2,\hdots,x_k)=\bigcap_{\ell=1}^k \G_{i_\ell}(x_\ell).
$
Assume that $y \in \G_i(x)$ for some $0 \le i \le \varepsilon(x)$ and let $z$ be a neighbour of $y$. Then, by the triangle inequality,
\begin{equation}
\label{desigualdadtriangular}
\partial(x, z) \in \left\lbrace i-1, i, i+1 \right\rbrace ,
\end{equation}
and so $z \in \G_{i-1}(x) \cup \G_i(x) \cup \G_{i+1}(x)$.
For $y \in \G_i(x)$ we therefore define the following numbers:
$$
a_i(x,y)=\left|\G_i(x)\cap \G(y) \right|, \hspace{0.4cm} b_i(x,y)=\left|\Gamma_{i+1}(x)\cap \G(y) \right|, \hspace{0.4cm} c_i(x,y)=\left|\G_{i-1}(x)\cap \G(y) \right|. 
$$
We say that $x \in X$ is {{\it distance-regularized} (or that $\G$ is {{\it distance-regular around $x$}) if the numbers $a_i(x,y),  b_i(x,y)$ and $c_i(x,y)$ do not depend on the choice of $y \in \G_i(x) \; (0 \le i \le \varepsilon(x))$. In this case, the numbers $a_i(x,y),  b_i(x,y)$ and $c_i(x,y)$ are simply denoted  by $a_i(x),  b_i(x)$ and $c_i(x)$ respectively, and are called the {{\it intersection numbers of $x$}. Observe that if $x$ is distance-regularized and $\varepsilon(x)=d$, then $a_0(x)=c_0(x)=b_d(x)=0$, $b_0(x)= |\G(x)|$ and $c_1(x)=1$. Note also that for every $1\leq i \leq d$ we have that $b_{i-1}(x)>0$ and $c_i(x)>0$, and that $a_i(x)=0$ if $\G$ is bipartite. For convenience we define $c_i(x)=b_i(x)=0$ for $i < 0$ and $i > d$.
			
A connected graph in which every vertex is distance-regularized is called a {\it distance-regularized} graph. A special case of such graphs are {\it distance-regular} graphs where all vertices have the same intersection array. Other examples are bipartite graphs in which vertices in the same color partition have the same intersection array, but which are not distance-regular. We call these graphs {\it distance-biregular}. It turns out that every distance-regularized graph is either distance-regular or distance-biregular (see \cite{GST}).  

A connected bipartite graph $\G$ with color partitions $Y$ and $Y'$ is called {\it distance-semiregular with respect to $Y$} if it is distance-regular around all vertices in $Y$, with the same parameters (i.e., there exist scalars $b_i$ and $c_i$ such that $b_i(x,y)=b_i$ and $c_i(x,y)=c_i$ for each $x\in Y$ and $y\in \G_i(x)$). In this case, $\G$ is biregular: each vertex in $Y$ has valency $b_0$ and each vertex in $Y'$ has valency equal to $b_1+1$. Note that every distance-biregular graph is distance-semiregular with respect to both color partitions $Y$ and $Y'$.

 The incidence graph of a design $\mathcal{D}=(\mathcal{P}, \mathcal{B}, \mathcal{I})$ is a $(\mathcal{P}, \mathcal{B})$-bipartite graph where the point $x \in \mathcal{P}$ is adjacent to the block $B\in\mathcal{B}$ if and only if $x$ is incident with $B$. In this case, we observe that all points have eccentricity $D$ but the eccentricity of the blocks are not necessarily the same. If $D=1$, we observe there exists a one-to-one correspondence between the incidence graph of {$1$-$(1,1,b)$} designs and bipartite distance-regularized graphs with vertices of eccentricity $1$ (complete bipartite graphs $K_{1,b}$ with $b\geq 1$). If $D=2$, it is clear that there exists a one-to-one correspondence between the incidence graphs of $2$-$(v,v,b)$ designs and bipartite distance-regularized graphs with vertices of eccentricy $2$ (complete bipartite graphs  $K_{v,b}$ with $v\geq 2$, $b\geq 1$). The incidence graphs of symmetric $2$-designs are pricesely bipartite distance-regular graphs with vertices of eccentricity $3$, which are well studied (see \cite{BCN, DKT}). The properties of incidence graphs of non-symmetric $2$-designs were studied in \cite{FR}.  In {\cite[Theorem 5.2]{FR}}, it is shown that there is a one-to-one correspondence between the incidence graphs of $2$-designs and distance-semiregular graphs with distance-regularized vertices of eccentricity $3$. Moreover, it turns out that quasi-symmetric $2$-designs with one intersection number zero are exactly distance-biregular graphs with $D=3$ where every block has eccentricity $D'=4$. In this paper, we will characterize incidence structures whose incidence graph is a bipartite distance-regularized graph with vertices of eccentricity $4$. 

The paper is organized as follows. In the next section we give some properties of bipartite distance-regularized graphs. In Section \ref{sec:spbibd} we introduce the concept of special partially balanced incomplete block designs. Then, we relate bipartite distance-regularized graphs with incidence graphs of special partially balanced incomplete block designs in Section \ref{sec:incspbibd}. Finally, in Section \ref{sec:dsg} we prove that there is a one-to-one correspondence between the incidence graphs of  quasi-symmetric SPBIBDs with parameters $(v,b,r,k, \lambda_1,0)$ of type $(k-1,t)$ with intersection numbers $x=0$ and $y>0$, where $0< y\leq t<k$ , and bipartite distance-regularized graphs with $D=D'=4$.\\

\section{Bipartite distance-regularized graphs} \label{sec:dbg}

In this section we recall some results about bipartite distance-regularized graphs which we will find useful later in the article.

Let $\G$ denote a $(Y,Y')$-bipartite distance-regularized graph with vertex set $X$. By \cite{GST} we observe that  $\G$ is either a bipartite distance-regular graph ($\G$ is regular and all of its vertices have the same intersection numbers) or $\G$ is a distance-biregular graph ($\G$ is not regular and vertices of the same bipartite class have the same intersection numbers).  

Pick now $x \in X$ and let $\varepsilon(x)$ denote the eccentricity of $x$. Since $\G$ is bipartite, we have $a_i(x)=0$ for $0 \le i \le \varepsilon(x)$. Note that all vertices from $Y$ ($Y'$, respectively) have the same eccentricity. We denote this common eccentricity by $D$ ($D'$, respectively). We also observe that $|D-D'| \le 1$ and the diameter of $\G$ equals $\max\{D,D'\}$. In addition, all vertices from $Y$ ($Y'$, respectively) have the same valency $k$ ($k'$, respectively). For $x \in Y$, $y \in Y'$ and an integer $i$ we abbreviate $c_i := c_i(x)$, $b_i := b_i(x)$, $c'_i := c_i(y)$ and $b'_i := b_i(y)$. 
 Observe that for $0 \leq i \leq D$, we have that $c_i+b_i=k$ if $i$ is even, while $c_i+b_i=k'$ if $i$ is odd. Moreover, if $k=k'$, it is not hard to see that $D=D'$ and the scalars $b_i=b_i'$, $c_i=c_i'$; see also \cite[Lemma 1]{DC}.

\begin{lemma}
\label{pi2iY}
Let $\G$ denote a $(Y,Y')$-bipartite distance-regularized graph. Let $D$ denote the eccentricity of vertices from $Y$ and assume $D\geq 3$.  Pick a vertex $x\in Y$. For every integer $i \; (2 \leq i \leq D-1)$ and for $z \in \G_i(x)$,  the number of vertices which are at distance $2$ from  $x \in Y$ and at distance $i$ from $z$ depends only on the bipartite part $Y$ but not on the choice of $x \in Y$ and $z \in \G_i(x)$.  Moreover, 
	$$|\G_2(x)\cap\G_i(z)|=
	\left\{\begin{array}{rl}
		\displaystyle \frac{b_i(c_{i+1}-1)+c_i(b_{i-1}-1)}{c_2} & \mbox{if } i \mbox{ is even, } \displaystyle\\
		\displaystyle \frac{b'_i(c'_{i+1}-1)+c'_i(b'_{i-1}-1)}{c_2} & \mbox{if } i \mbox{ is odd. } \displaystyle
	\end{array}\right. $$ 
	In this case, we simply write $p^i_{2,i}:=p^i_{2,i}(Y)=|\G_2(x)\cap\G_i(z)|$.
\end{lemma}

\begin{proof}
	Recall that every bipartite distance-regularized graph is either a bipartite distance-regular graph or a distance-biregular graph. The proof now immediately follows from \cite[Lemma~4.1.7]{BCN} and \cite[Lemma~3.1]{FR}. \end{proof}

Let $\G$ denote a bipartite graph with vertex set $X$, color partitions $Y$, $Y'$,  and assume that every vertex in $Y$ has eccentricity $D\ge 3$. For $z\in X$ and a non-negative integer $i$, recall that $\G_{i}(z)$ denotes the set of vertices in $X$ that are at distance $i$ from $z$. Graph $\G$ is {\it almost $2$-$Y$-homogeneous} whenever for all $i \; (1\leq i \leq D-2)$ and for all $x\in Y$, $y \in \G_2(x)$ and $z \in \G_{i}(x)\cap\G_i(y)$, the number of common neighbours of $x$ and $y$ which are at distance $i-1$ from $z$ is independent of the choice of $x$, $y$ and $z$. In addition, if the above condition holds also for the case $i=D-1$, then we say that $\G$ is {\it $2$-$Y$-homogeneous}. For $(Y,Y')$-bipartite distance-regularized graphs, we remark that the (almost) $2$-$Y$-homogeneous condition generalizes the notion of (almost) $2$-homogeneous distance-regular graphs which was well-studied by Curtin and Nomura; see for more details \cite{BC, CB, BCT, nomura1995spin}. Moreover, the (almost) $2$-$Y$-homogeneous condition in distance-biregular graphs was recently studied in \cite{FBPS} where the authors found necessary and sufficient conditions on the intersection array of $\G$ for which the graph is (almost) $2$-$Y$-homogeneous.

  Assume that $\G$ is a  $(Y,Y')$-bipartite distance-regularized graph. Note that, if $D=2$ then $\G$ is a complete bipartite graph with $D'\in \{1,2\}$ (if $D'=3$, then for any $x\in Y'$, $\G_3(x)\subseteq Y$, which yields $D\ge 3$, a contradiction). Moreover, if $D=2$ then $\G$ is $2$-$Y$-homogeneous by definition. For the rest of the paper we assume that $D\ge 3$ (which also yields $D'\geq 3$ - just use the same argument as in the previous sentence).
  
  Suppose that $\G$ has vertices of valency $2$. If $\G$ is regular then $\G$ is a cycle of even length and so, $\G$ is $2$-$Y$-homogeneous and $2$-$Y'$-homogeneous. Otherwise, by \cite[Corollary~3.5]{MST}, a graph $\G$ with  vertices of valency $2$ is distance-biregular if and only if $\G$ is either {the complete bipartite graph} $\G=K_{2,r}$, or $\G$ is the subdivision graph of a $(\kappa,g)$-cage graph. In \cite[Section~4]{FBPS}, it was shown that a $(Y,Y')$-distance-biregular graph with $k'=2$ is $2$-$Y$-homogeneous, and  some combinatorial properties of such graphs were given. {We then focus our attention on $(Y,Y')$-bipartite distance-regularized graphs with $k'\geq 3$.}

We define certain scalars $\Delta_i$ $(2\le i\le \min\left\lbrace D-1, D'-1 \right\rbrace)$, which can be computed from the intersection array of a given bipartite distance-regularized graph. These scalars play an important role, since from their values we can decide if a given bipartite distance-regularized graph is (almost) $2$-$Y$-homogeneous or not. 

\begin{definition}\label{def}{\rm
		Let $\G$ denote a $(Y,Y')$-bipartite distance-regularized graph $\G$. Let $D$ ($D'$, respectively) denote the eccentricity of vertices from $Y$ ($Y'$, respectively). Pick $i$ $(1\leq i \leq \min\left\lbrace D-1, D'-1 \right\rbrace)$, and with reference to Lemma \ref{pi2iY}, define the scalar $\Delta_i=\Delta_i(Y)$ in the following way:
		$$
		\Delta_{i}=
		\left\{\begin{array}{rl}
			\displaystyle(b_{i-1}-1)(c_{i+1}-1)-p^i_{2,i}(c_2'-1) & \mbox{if } i \mbox{ is even,}\\
			\displaystyle (b'_{i-1}-1)(c'_{i+1}-1)-p^i_{2,i}(c_2'-1) & \mbox{if } i \mbox{ is odd.} 
		\end{array}\right. 
		$$
}\end{definition}

 We end this section pointing out the following results which we will find useful later to decide if a given bipartite distance-regularized graph has the (almost) $2$-$Y$-homogeneous condition. 
 
\begin{theorem}
	\label{Ek}
	With reference to Definition~\ref{def}, let $\G$ denote a $(Y,Y')$-bipartite distance-regularized graph with $k'\ge 3$ and $D\ge 3$. Then the following are equivalent.
	\begin{enumerate}[label=(\roman*),font=\rm]
		\item $\Delta_i=0$ $(2\le i\le \min\{D-1,D'-1\})$.
		\item For every $i$ $(2\le i\le D-1)$, there exist $x \in Y$ and $y\in\G_2(x)$ such that for all $z\in\G_{i,i}(x,y)$ the number $|\G_{1,1,i-1}(x,y,z)|$ is independent of the choice of $z$.
		\item $\G$ is $2$-$Y$-homogeneous.
	\end{enumerate}
\end{theorem}

\begin{proof}
	Immediate from the definition of the $2$-$Y$-homogeneous condition, \cite[Theorem~13]{BC} and \cite[Corollary~7.3]{FBPS}. 
\end{proof}

\begin{theorem}
	\label{Ekk}
	With reference to Definition~\ref{def}, let $\G$ denote a $(Y,Y')$-bipartite distance-regularized graph with $k'\ge 3$ and $D\ge 3$. Then the following are equivalent.
	\begin{enumerate}[label=(\roman*),font=\rm]
		\item $\Delta_i=0$ $(2\le i\le D-2)$.
		\item For every $i$ $(2\le i\le D-2)$, there exist $x \in Y$ and $y\in\G_2(x)$ such that for all $z\in\G_{i,i}(x,y)$ the number $|\G_{1,1,i-1}(x,y,z)|$ is independent of the choice of $z$.
		\item $\G$ is almost $2$-$Y$-homogeneous.
	\end{enumerate}
\end{theorem}

\begin{proof}
	Immediate from the definition of the $2$-$Y$-homogeneous condition, \cite[Theorem~13]{BC} and \cite[Corollary~7.4]{FBPS}. 
\end{proof}

\section{Special partially balanced incomplete block designs}
\label{sec:spbibd}

Let $\mathcal{D}$ be a $1$-$(v,b,r,k,\lambda)$ design and let $(s,t)$ be a pair of non-negative integers. A flag (a non-flag) of $\mathcal{D}$ is a point-block pair $(p, B)$ such that $p\in B$ ($p \notin B$). We say that $\mathcal{D}$ is a {\it special partially balanced incomplete block design} (SPBIBD for short) of type $(s,t)$ if there are constants $\lambda_1$ and $\lambda_2$ with the following properties: 
\begin{enumerate}[label=(\roman*),font=\rm]
	\item Any two points are contained in either $\lambda_1$ or $\lambda_2$ blocks. 
	\item If a point-block pair $(p,B)$ is a flag, then the number of points in $B$ which occur with $p$ in $\lambda_1$ blocks is $s$.  
	\item   If a point-block pair $(p,B)$ is a non-flag, then the number of points in $B$ which occur with $p$ in $\lambda_1$ blocks is $t$.
	\end{enumerate}
In this case, we say that $\mathcal{D}$ is a $(v,b,r,k,\lambda_1, \lambda_2)$ SPBIBD of type $(s,t)$.

  The intersection numbers of a $1$-$(v, k, \lambda)$ design are the cardinalities of the intersection of any two distinct blocks. 
  Let $x$ and $y$ be non-negative integers with $x< y$. A design $\mathcal{D}$ is called a (proper) {\it quasi-symmetric design} with intersection numbers $x$ and $y$ if any two distinct blocks of $\mathcal{D}$ intersect in either $x$ or $y$ points, and both intersection numbers are realized. That is, if $|B\cap B'|\in \{x,y\}$ for any pair of distinct blocks $B, B'$ and both intersection numbers occur. 

Recall that the dual of a design $\mathcal{D}=(\mathcal{P}, \mathcal{B}, \mathcal{I})$ is the structure $\mathcal{D'}=(\mathcal{B}, \mathcal{P}, \mathcal{I'})$  such that $(B,p)\in  \mathcal{I'}$ if and only if $(p,B)\in  \mathcal{I}$  for every $(p,B)\in \mathcal{B} \times \mathcal{P}$. The dual of a quasi-symmetric SPBIBD is an SPBIBD (see \cite[Theorem 4.39]{SS-QS}).

A {\it partial geometry $(r, k, t)$} is a $1$-$(v,b,r, k, \lambda)$ design such that:
\begin{itemize}
\item any two points are incident with at most one block
\item for every non-flag $(x,B)$, there exist exactly $t$ blocks
that are incident with $x$ and intersect $B$ $(1 \leq t \leq r, 1 \leq t \leq k)$.
\end{itemize} 

Notice that partial geometries are exactly  quasi-symmetric $(v,b,r,k,1,0)$ SPBIBDs of type $(k-1,t)$  with intersection numbers $x=0$ and $y=1$. A partial geometry for which $t=1$ is called a {\it generalized quadrangle}.

Let $\mathcal{D}$ be a $(v,b,r,k,\lambda_1,0)$ SPBIBD of type $(k-1,t)$. If $r=b$, then every point belongs to every block, which is a trivial case that we will not consider. Therefore, we will assume $r<b$. Also, if $t=0$, then every two different blocks do not interesect, and if $t=k$, then $\mathcal{D}$ is a $2$-design.

\begin{lemma} \label{cases}
	Let $\mathcal{D}=(\mathcal{P}, \mathcal{B}, \mathcal{I})$ be a quasi-symmetric $(v,b,r,k,\lambda_1,0)$ SPBIBD of type $(k-1,t)$  with intersection numbers $x=0$ and $y>0$. Then $y\leq t<k$. If $y>1$, then $t>y$ and $\lambda_1< t\lambda_1/y<r$.
 
\end{lemma} 

\begin{proof}
 Let $(p,B)$  be a non-flag and let $B_1$ be a block such that $p \in  B_1$  and $B_1$ intersects $B$ in $y$ points. From this it follows that the number of points in $B$ which occur with $p$ in $\lambda_1$ blocks is greater or equal than $y$, that is $y\leq t$. Also $t<k$, since from $t=k$ it follows that every two points are together in $\lambda_1$ blocks.

  Let $y>1$ and $y=t$. For a non-flag $(p,B)$  let $\{p_1, p_2,...,p_t\} \subset B$ be the set of points in $B$ which occur with $p$ in $\lambda_1$ blocks.   Every block $B_1$ containing $p$ and $p_1$ will intersect $B$ in points $p_1, p_2,...,p_t$ , and, since $y=t$, it follows that $\lambda_1=1$.  Also, in that case there are at least $2=\lambda_1+1$ blocks ($B$ and $B_1$)  containing $p_1$  and $p_2$ , which is not possible. It follows that $t>y$. Furthermore, for a non-flag $(p,B)$  there are exactly $t\lambda_1/y$ blocks containing $p$  and intersecting $B$. Since $t>y,$ it follows $\lambda_1< t\lambda_1/y$. Also $t\lambda_1/y<r$, because in the contrary every two block will intersect.
 \end{proof}


 \begin{theorem}\label{cases2}
For the quasi-symmetric $(v,b,r,k,\lambda_1,0)$ SPBIBD of type $(k-1,t)$  with intersection numbers $x=0$ and $y>1$ it follows that $k\geq 4$ and $r\geq 4$.
\end{theorem}
\begin{proof}
By Lemma \ref{cases}, from $y>1$ we have $t>2$ and $k>3$. 
Furthermore, $0<\lambda_1<t\lambda_1/y<r$ implies that $r \geq 3$. If $r=3$ then $\lambda_1=1$, so $y=1$. We conclude that $r\geq 4.$
\end{proof}

\section{The incidence graph of a SPBIBD}
\label{sec:incspbibd}

The incidence graph of a design $\mathcal{D}=(\mathcal{P}, \mathcal{B}, \mathcal{I})$ is a $(\mathcal{P}, \mathcal{B})$-bipartite graph where the point $x \in \mathcal{P}$ is adjacent to the block $B\in\mathcal{B}$ if and only if $x$ is incident with $B$. In this section, we mention some properties about the incidence graph of certain SPBIBDs. We also study the $2$-$\mathcal{P}$-homogeneous and $2$-$\mathcal{B}$-homogeneous conditions in these graphs.

\begin{lemma}\label{lemma3}
	Let $\mathcal{D}=(\mathcal{P}, \mathcal{B}, \mathcal{I})$ be a $1$-$(v,b,r,k,\lambda_1, 0)$  SPBIBD of type $(k-1,t)$ where $0<t<k$. Let $\G$ denote the incidence graph of $\mathcal{D}$. Then, every vertex $p \in \mathcal{P}$ has eccentricity $4$ in $\G$. 
\end{lemma}

\begin{proof}
	  Notice that for any $p\in \PP$, there are $r$ blocks containing $p$ and so, which are at distance $1$ from $p$. Moreover, since $r<b$ there exists a block $B \in \BB$ such that $(p,B)$ is not a flag. 
	
	Let $S$ denote the set of all points incident with $B$ which are at distance greater than $2$ from $p$. We observe that $S$ is nonempty and has size $k-t<k$.  Pick $q\in S$. Since $\G$ is bipartite, in this case we have that $\partial(p,q)\geq 4$. Note also that there are $t$ points in $B$ which are at distance $2$ from $p$. Let $w$ be such a point. For a block $B'$ containing both $p$ and $w$, it follows that $\left[ p, B', w, B, q\right]$ is a $pq$-path of length $4$, meaning that $\partial(p,q)=4$ and $\partial(p,B)=3$. 	
	Hence, from the above comments, every point $p\in \PP$ has eccentricity $4$. 
\end{proof}

\begin{lemma}\label{lemma11}
	Let $\mathcal{D}=(\mathcal{P}, \mathcal{B}, \mathcal{I})$ be a $(v,b,r,k,\lambda_1,0)$ SPBIBD of type $(k-1,t)$, where $0<t<k.$ Let $\G$ denote the incidence graph of $\mathcal{D}$.  Then, every vertex $p \in \mathcal{P}$ is distance-regularized. Moreover, $\G$ is distance-semiregular with respect to $\mathcal{P}$ with the following intersection numbers:
	\begin{eqnarray}
		c_0&=&0, \hspace{0.5cm} c_{1}=1, \hspace{1.2cm} c_{2}=\lambda_1, \hspace{1.2cm} c_{3}=t, \hspace{1.75cm} c_{4}=r. \nonumber \\
		b_0&=&r, \hspace{0.5cm} b_{1}=k-1, \hspace{0.5cm} b_{2}=r-\lambda_1, \hspace{0.5cm} b_{3}=k-t \hspace{1.2cm} b_{4}=0. \nonumber
	\end{eqnarray} 
\end{lemma}

\begin{proof}
	 Note that $\G$ is bipartite with bipartitions $\mathcal{P}$ and $\mathcal{B}$. As every point is contained in $r$ blocks and every block has size $k$, it is easy to see that $\G$ is $(r,k)$-biregular. We notice that every pair of points either are contained in exactly $\lambda_1$ blocks or are not contained in any block. Moreover, by Lemma~\ref{lemma3} every point in $\mathcal{P}$ has eccentricity equal to $4$. 
	
	We will now show that every point in $\mathcal{P}$ is distance-regularized. Pick $p\in \mathcal{P}$. For every $z\in \G_i(p) \; (0 \leq i \leq 4)$, consider the numbers $b_i(p,z)$, $a_i(p,z)$ and $c_i(p,z)$. Recall that $a_i(p,z)=0$ as the graph $\G$ is bipartite. Note that $b_0(p,z)=|\G(p)|$, $c_0(p,z)=0$, $c_1(p,z)=1$ and $b_4(p,z)=0$. As every block has size $k$, every block in $\G(p)$ contains $k-1$ points different from $p$ and so, $b_1(p,z)=k-1$. Similarly, as every pair of points that appears in a block is contained in $\lambda_1$ blocks we have $c_2(p,z)=\lambda_1$ and, since every point is contained in $r$ blocks, we have $b_2(p,z)=r-\lambda_1$. Moreover, for a block $z\in \G_3(p)$ we have $t$ points in $z\in \BB$ which occur with $p$ in $\lambda_1$ blocks. This yields $c_3(p,z)=t$ and $b_3(p,z)=k-t$, since every block has size $k$. Furthermore, $c_4(p,z)=r$ as every point appears in $r$ blocks. Thus, the numbers $b_i(p,z)$, $a_i(p,z)$ and $c_i(p,z)$ do not depend on the choice of $z\in \G_i(p) \; (0 \leq i \leq 4)$, and $\G$ is distance-regular around $p$. It follows from the above comments that $\G$ is distance-semiregular with respect to $\mathcal{P}$. 
\end{proof}

\begin{lemma}\label{block}
	Let $\mathcal{D}=(\mathcal{P}, \mathcal{B}, \mathcal{I})$ be a quasi-symmetric $(v,b,r,k, \lambda_1,0)$ SPBIBD of type $(k-1,t)$  with intersection numbers $x=0$ and $y>0$. Then, for every $B \in \mathcal{B}$ there exist $B_1, B_2 \in \BB$ such that $|B\cap B_1|=0$ and $|B\cap B_2|=y$. 
\end{lemma}

\begin{proof}
	Let $B \in \mathcal{B}$. Since $k<v$, there exists $p \in \PP$ such that $(p, B)$ is a non-flag. Note also that there are $t$ points in $B$ which are at distance $2$ from $p$. Let $q$ be such a point. Then, there exists a block $B_2 \in \mathcal{B}$ such that $\{p,q\}\subseteq B_2$. Thus, $q \in B\cap B_2 $ and therefore, $|B\cap B_2|=y$. 
	
	Let $\mathcal{D}'$ denotes the dual od $\mathcal{D}$. Since $\mathcal{D}$ is quasi-symmetric, $\mathcal{D}'$ is an SPBIBD, let say of type $(s',t')$. Now, suppose that $B$ intersects all the other blocks in $\BB$. This means that in the incidence graph of $\mathcal{D}'$ there exists a point $p'$ which is at distance $2$ from any other point.
	If $p$ is contained in all blocks of $\mathcal{D}'$, then $k=v$, which is not possible. So, there exists a block $B'$ of $\mathcal{D}'$ such that the pair $(p,B')$ is a non-flag. Since $p'$ is at distance $2$ from any other point then the number of points in $B'$ that are at distance $2$ from $p'$ equals the size of a block. That is, $t'=r$. 
	
	Now, let $q'\neq p'$ be a point of $\mathcal{D}'$. Then, any other point $q''$ in $\mathcal{D}'$ is either in a block with $q''$ or not. In the first case, the distance between $q'$ and $q''$ is $2$. In the second case, there exists a block $B''$ of $\mathcal{D}'$ such that $q'' \in B''$ and the pair $(q', B'')$ is a non-flag. Since $t'$ equals the size od the block in $\mathcal{D}'$,  we have that $q''$ is at distance $2$ from $q'$. This shows that $q'$ is at distance two from any other point. Therefore, any two points of $\mathcal{D}'$ are at distance 2. In the other words, any two blocks of $\mathcal{D}$ intersect, which contradicts $x=0$. Hence, there must be a block $B_1\in \BB$ such that $B\cap B_1=\emptyset$. This finishes the proof.   
\end{proof}

\begin{corollary}\label{tnotr}
	Let $\mathcal{D}=(\mathcal{P}, \mathcal{B}, \mathcal{I})$ be a quasi-symmetric $(v,b,r,k,\lambda_1,0)$ SPBIBD of type $(k-1,t)$  with intersection numbers $x=0$ and $y>0$. Then $t<r$. 
\end{corollary}

\begin{proof}
	Assume that $t=r$. Pick $B\in \mathcal{B}.$ Then for every $p \in \mathcal{P}$ and $B' \in \mathcal{B}$ such that $(p,B)$ is a non-flag and $(p,B')$ is a flag, it follows that $\left| B\cap B'\right| =y$. Hence, there is no $B_1 \in \mathcal{B}$ such that $\left| B\cap B_1\right| =0$, contradicting Lemma \ref{block}.  
\end{proof}

\begin{lemma}\label{lemma4}
	Let $\mathcal{D}=(\mathcal{P}, \mathcal{B}, \mathcal{I})$ be a quasi-symmetric $(v,b,r,k,\lambda_1,0)$ SPBIBD of type $(k-1,t)$  with intersection numbers $x=0$ and $y>0$.  Let $\G$ denote the incidence graph of $\mathcal{D}$. Then, every vertex $B \in \mathcal{B}$ has eccentricity $4$ in $\G$. 
\end{lemma}

\begin{proof}
	Recall that $\G$ is the $(\mathcal{P}, \mathcal{B})$-bipartite graph where the point $ p \in \mathcal{P}$ is adjacent to the block $B\in\mathcal{B}$ if and only if  $p$ is incident with $B$. Moreover,  as $\mathcal{D}$ is quasi-symmetric, any two distinct blocks of $\mathcal{D}$ intersect in either $x=0$ or $y>0$ points, and both intersection numbers are realized. By Lemma~\ref{block} we also know that for every $B \in \mathcal{B}$ there is a blocks that intersect $B$ and another block that does not intersect $B$.
	
	Pick $B \in \mathcal{B}$ and let $B' \in \mathcal{B}$. If $B\cap B'\neq \emptyset$ then $B$ and $B'$ have a point in common, showing that $\partial(B,B')=2$. Assume next  that $B\cap B'=\emptyset$. Pick $p\in B$. Since $(p,B')$ is not a flag there exist $t$ points in $B'$ which are at distance $2$ from $p$. Let $w$ be such a point. Notice $w\notin B$ since $B$ and $B'$ are disjoint. So, for a block $B''$ containing both $p$ and $w$, it follows that $\left[ B, p, B'', w, B'\right]$ is a path of length $4$, meaning that $\partial(B,B')=4$. 
	
	Pick now $p \in \PP$ and $B\in \BB$. If $p\in B$ we have that $\partial(p,B)=1$. Suppose next that $p \notin B$.  Since $\G$ is bipartite we observe that $\partial(p,B)$ is odd and so, $\partial(p,B)\geq 3$. Moreover, by Lemma \ref{lemma3}, every point in $\PP$ has eccentricity $4$ and so  $\partial(p,B)\leq 4$. This shows that $\partial(p,B)= 3$ if $p$ does not belong to $B$. 	Thus, the eccentricity of any block in $\G$ is equal to $4$. 
\end{proof}

The following theorem characterizes incidence graphs of  a quasi-symmetric  $(v,b,r,k,\lambda_1,0)$ SPBIBD of type $(k-1,t)$  with intersection numbers $x=0$ and $y>0$.

\begin{theorem}\label{he}
	Let $\mathcal{D}=(\mathcal{P}, \mathcal{B}, \mathcal{I})$ be a quasi-symmetric $(v,b,r,k,\lambda_1,0)$ SPBIBD of type $(k-1,t)$  with intersection numbers $x=0$ and $y>0$.  Let $\G$ denote the incidence graph of $\mathcal{D}$. Then, $\G$ is a $(\mathcal{P}, \mathcal{B})$-bipartite distance-regularized graph. Moreover, every vertex $p \in \mathcal{P}$ has eccentricity equals $4$ and the following intersection numbers:
	\begin{eqnarray}
		c_0&=&0, \hspace{0.5cm} c_{1}=1, \hspace{1.2cm} c_{2}=\lambda_1, \hspace{1.2cm} c_{3}=t, \hspace{1.75cm} c_{4}=r. \nonumber \\
		b_0&=&r, \hspace{0.5cm} b_{1}=k-1, \hspace{0.5cm} b_{2}=r-\lambda_1, \hspace{0.5cm} b_{3}=k-t \hspace{1.2cm} b_{4}=0. \nonumber
	\end{eqnarray} 
	In addition, every vertex $B \in \mathcal{B}$ has eccentricity equals $4$ and the following intersection numbers:
	\begin{eqnarray}
		c_0'&=&0, \hspace{0.5cm} c'_{1}=1, \hspace{1.2cm} c'_{2}=y, \hspace{1.25cm} c'_{3}=\frac{t\lambda_1}{y}, \hspace{1.2cm} c'_{4}=k. \nonumber \\
		b_0'&=&k, \hspace{0.5cm} b'_{1}=r-1, \hspace{0.5cm} b_{2}'=k-y, \hspace{0.5cm} b'_{3}=r-\frac{t\lambda_1}{y}, \hspace{0.5cm} b'_{4}=0. \nonumber 
	\end{eqnarray} 	
	
\end{theorem}

\begin{proof}
	Since  $\mathcal{D}$ is a $(v,b,r,k,\lambda_1,0)$ SPBIBD of type $(k-1,t)$  where $0< t<k$, by Lemma \ref{lemma11}, every vertex in $\mathcal{P}$ has eccentricity $4$ and $\G$ is distance-semiregular with respect to $\mathcal{P}$, where the intersection numbers for every $p\in \mathcal{P}$ can be computed as in the proof of Lemma \ref{lemma11}. Moreover, by Lemma~\ref{lemma4}  the eccentricity of every block in $\G$ equals $4$. 
	
	We will now show that every block in $\G$ is distance-regularized. Pick $B\in \mathcal{B}$. For every $z\in \G_i(B) \; (0 \leq i \leq 4)$, consider the numbers $b'_i(B,z)$, $a'_i(B,z)$ and $c'_i(B,z)$. Recall that $a'_i(B,z)=0$, as the graph $\G$ is bipartite. Note that $b'_0(B,z)=|\G(B)|$, $c'_0(B,z)=0$, $c'_1(B,z)=1$ and $b'_4(B,z)=0$. As every two blocks which are at distance $2$ have $y$ common neighbours, we have $c'_2(B,z)=y$.  Suppose for the moment that $z\in \G_3(B)$. We will count the number of paths of length $3$ between $B$ and $z$ in two different ways. Firstly, observe that $z$ has $c'_3(B,z)$ neighbours in $\G_2(B)$ and each of these neighbours is adjacent to exactly $y$ vertices in $\G(B)$, since every two intersecting blocks have $y$ points in common. Secondly, note that there are $t$ points in $B$ which occur with $p$ in $\lambda_1$ blocks. Thus, we have $c'_3(B,z)y=t\lambda_1$. Also, as each block contains $k$ points, we have $c'_4(B,z)=k$. Furthermore, since every point occurs in $r$ blocks and each block has size $k$, it follows from the above comments the numbers $b'_i(B,z)$ and $c'_i(B,z)$ do not depend on the choice of $z\in \G_i(B) \; (0 \leq i \leq 4)$. Hence, $\G$ is distance-regular around $B$. If $k=r$ then by \cite[Lemma 1]{DC} we have that $\lambda_1=y$ and so, $\G$ is a bipartite distance-regular graph with diameter $4$. Otherwise, $\G$ is a distance-biregular graph with $D=D'=4$. This finishes the proof. 
\end{proof}

\subsection{The (almost) $2$-$\mathcal{P}$-homogeneous condition}

According to Theorem~\ref{he}, the incidence graph of certain SPBIBDs is a $(\mathcal{P}, \mathcal{B})$-bipartite distance-regularized graph with all vertices of eccentricity equal to four. Here we explore the (almost) $2$-$\mathcal{P}$-homogeneous condition of the incidence graph of such SPBIBD.

\begin{proposition}\label{k2}
	Let $\G$ denote the incidence graph of a $\mathcal{D}=(\mathcal{P}, \mathcal{B}, \mathcal{I})$  quasi-symmetric $(v,b,r,k,\lambda_1,0)$ SPBIBD of type $(k-1,t)$  with intersection numbers $x=0$ and $y=1$. The following \rm{(i)},\rm{(ii)} are equivalent:
	\begin{enumerate}[label=(\roman*),font=\rm]
		\item Each block has size $2$.
		\item $\G$ is isomorphic to the subdivision graph of a complete bipartite graph $K_{r,r}=(X,\mathcal{R})$ with $r \geq 2$ and $X=\mathcal{P}$.
	\end{enumerate}
In this case, $\G$ is $2$-$\mathcal{P}$-homogeneous. 
\end{proposition}

\begin{proof}
	Suppose that each block has  size $2$. We observe that $k=2$ and $y=\lambda_1=1$. By Lemma~\ref{cases} we also have that $1=y\leq t <k=2$ and so, $t=1$. By Theorem~\ref{he}, we thus have that $\G$ is a $(\mathcal{P}, \mathcal{B})$-bipartite distance-regularized graph where every point and every block has eccentricity equal to four. Since every vertex has eccentricity $4$ we have that $r>1$. If $r=2$ then $\G$ is a distance-regular graph of diameter $4$ and valency $2$; i.e. $\G$ is the subdivision graph of the complete bipartite graph $K_{2,2}$. Otherwise, $\G$ is a $(\mathcal{P}, \mathcal{B})$-distance-biregular graph. Moreover, the intersection numbers of every point are $b_0=r$, $c_i=1 \; (1\leq i \leq 3)$ and $c_4=r$. It follows from \cite[Theorem~4.2]{FBPS}  that $\G$ is isomorphic to the subdivision graph of an $(r,4)$-cage, i.e. a complete bipartite graph $K_{r,r}=(X,\mathcal{R})$ with $r \geq 2$ and $X=\mathcal{P}$. Conversely, we notice that $\G$ is a $(\mathcal{P},\mathcal{B})$-bipartite graph with $|\mathcal{P}|=2r$ and $|\mathcal{B}|=r^2$ where every point in $\mathcal{P}$ has valency $r$ while every block in $\mathcal{B}$ has valency $2$. To prove our last claim, observe that from \cite[Theorem 4.2]{FBPS}, for all $p\in\mathcal{P}$ and $q \in\G_2(p)$, the sets $\G_{2,2}(p, q)$ are empty. Thus, for all $i \; (2 \leq i \leq 3)$
	and for all $p\in\mathcal{P}$, $q \in\G_2(p)$, and  $z \in\G_{2,2}(p, q)$, the number $|\G_{i-1}(z) \cap \G_{1,1}(x, y)|$ equals $0$. This shows that $\G$ is $2$-$\mathcal{P}$-homogeneous.
\end{proof}

\begin{proposition}\label{P1}
	Let $\G$ denote the incidence graph of a $\mathcal{D}=(\mathcal{P}, \mathcal{B}, \mathcal{I})$  quasi-symmetric $(v,b,r,k,\lambda_1,0)$ SPBIBD of type $(k-1,t)$  with intersection numbers $x=0$ and $y=1$. Assume that $k\geq 3$. Then the following hold:
	\begin{enumerate}[label=(\roman*),font=\rm]
		\item $\G$ is almost $2$-$\mathcal{P}$-homogeneous if and only if $\mathcal{D}$ is a generalized quadrangle.
		\item $\G$ is not $2$-$\mathcal{P}$-homogeneous.
	\end{enumerate}
	\end{proposition}

\begin{proof}
	By Theorem~\ref{he}, we observe that $\G$ is a $(\mathcal{P}, \mathcal{B})$-bipartite distance-regularized graph where every point and every block has eccentricity equal to four. Moreover, the intersection arrays of $\G$ can be computed in term of the parameters of $\mathcal{D}$. Therefore, to analyze the (almost) $2$-$\mathcal{P}$-homogeneous condition of $\G$, it suffices to compute the scalars $\Delta_2(\mathcal{P})$ and $\Delta_3(\mathcal{P})$, as defined in Definition~\ref{def}. For $(0\leq i \leq 4)$, let $c_i, b_i$ and $c'_i, b'_i$ denote the intersection numbers of the points and the blocks, respectively, as shown in Theorem~\ref{he}. Since $c_2'=y=1$, it turns out that
	\begin{eqnarray}
		\Delta_2(\mathcal{P})&=&(b_1-1)(c_3-1)=(k-2)(t-1) \label{delta2p}, \\ \Delta_3(\mathcal{P})&=&(b_2'-1)(c_4'-1)=(k-2)(k-1) \label{delta3p}.
	\end{eqnarray}   
By Theorem~\ref{Ekk}, we have that $\G$ is almost $2$-$\mathcal{P}$-homogeneous if and only if $\Delta_2(\mathcal{P})=0$. If $\mathcal{D}$ is a generalized quadrangle then $t=1$ which shows that $\Delta_2(\mathcal{P})=0$. Conversely, if $\Delta_2(\mathcal{P})=0$ then $t=1$, since $k>2$. This means that  $\mathcal{D}$ is a generalized quadrangle if and only if $\Delta_2(\mathcal{P})=0$. So, (i) follows. Moreover, from \eqref{delta3p} the scalar $\Delta_3(\mathcal{P})>0$ and so, by Theorem~\ref{Ek} we have that $\G$ is not $2$-$\mathcal{P}$-homogeneous. 
\end{proof}


\begin{proposition}\label{k3}
	Let $\G$ denote the incidence graph of a $\mathcal{D}=(\mathcal{P}, \mathcal{B}, \mathcal{I})$  quasi-symmetric $(v,b,r,k,\lambda_1,0)$ SPBIBD of type $(k-1,t)$  with intersection numbers $x=0$ and $y>1$. In this case, $\G$ is almost $2$-$\mathcal{P}$-homogeneous if and only if $$k=\frac{(y-1)(r-\lambda_1)(t-1)}{\lambda_1(t-y)}+2.$$
\end{proposition}

\begin{proof}
	By Theorem~\ref{he}, we observe that $\G$ is a $(\mathcal{P}, \mathcal{B})$-bipartite distance-regularized graph where every point and every block has eccentricity equal to four. Moreover, the intersection arrays of $\G$ can be computed in term of the parameters of $\mathcal{D}$. Therefore, to analyze the (almost) $2$-$\mathcal{P}$-homogeneous condition of $\G$, it suffices to compute the scalars $p^2_{2,2}(\mathcal{P})$, as shown in Lemma~\ref{pi2iY}, and $\Delta_2(\mathcal{P})$, as defined in Definition~\ref{def}. For $(0\leq i \leq 4)$, let $c_i, b_i$ and $c'_i, b'_i$ denote the intersection numbers of the points and the blocks, respectively, as shown in Theorem~\ref{he}. Therefore, by Lemma~\ref{pi2iY} it follows that 
	\begin{equation}\label{eqq1}
		p^2_{2,2}(\mathcal{P})=\frac{(r-\lambda_1)(t-1)+\lambda_1(k-2)}{\lambda_1}. 
	\end{equation}
	Then, by Definition~\ref{def} and \eqref{eqq1} we have that 
		\begin{equation}\label{eqq2}
		\Delta_2(\mathcal{P})=(k-2)(t-1)-(y-1)\frac{(r-\lambda_1)(t-1)+\lambda_1(k-2)}{\lambda_1}. 
	\end{equation}
By Theorem~\ref{Ekk}, we have that $\G$ is almost $2$-$\mathcal{P}$-homogeneous if and only if $\Delta_2(\mathcal{P})=0$.
From \eqref{eqq2}, it is easy to see that $\Delta_2(\mathcal{P})=0$ if and only if 
\begin{eqnarray}\label{eqq3}
	\lambda_1(k-2)(t-y)=(y-1)(r-\lambda_1)(t-1).
\end{eqnarray}
Since $y>1$, by Lemma \ref{cases} and Theorem \ref{cases2} we also observe that $t>y$ and $k\geq 4$. Therefore, from \eqref{eqq3} it is easy to see that $$k-2=\frac{(y-1)(r-\lambda_1)(t-1)}{\lambda_1(t-y)}.$$
The claim now immediately follows from Theorem~\ref{Ekk}. 
\end{proof}

\begin{proposition}\label{k4}
	Let $\G$ denote the incidence graph of a $\mathcal{D}=(\mathcal{P}, \mathcal{B}, \mathcal{I})$  quasi-symmetric $(v,b,r,k,\lambda_1,0)$ SPBIBD of type $(k-1,t)$  with intersection numbers $x=0$ and $y>1$. In this case, $\G$ is $2$-$\mathcal{P}$-homogeneous if and only if $$k=\frac{(y-1)(r-\lambda_1)(t-1)}{\lambda_1(t-y)}+2, \hspace{0.5cm} t=\frac{(k-1)\left[r(y-1)-\lambda_1(k-y-1) \right] }{\lambda_1(y-1)}.$$
\end{proposition}

\begin{proof}
		By Theorem~\ref{he}, we observe that $\G$ is a $(\mathcal{P}, \mathcal{B})$-bipartite distance-regularized graph where every point and every block has eccentricity equal to four. Moreover, the intersection arrays of $\G$ can be computed in term of the parameters of $\mathcal{D}$. Therefore, to analyze the $2$-$\mathcal{P}$-homogeneous condition of $\G$, it suffices to compute the scalars $p^3_{2,3}(\mathcal{P})$, as shown in Lemma~\ref{pi2iY}, and $\Delta_3(\mathcal{P})$, as defined in Definition~\ref{def}. For $(0\leq i \leq 4)$, let $c_i, b_i$ and $c'_i, b'_i$ denote the intersection numbers of the points and the blocks, respectively, as shown in Theorem~\ref{he}. Therefore, by Lemma~\ref{pi2iY} it follows that 
	\begin{equation}\label{eqq11}
		p^3_{2,3}(\mathcal{P})= \frac{\left( r-\frac{t\lambda_1}{y}\right) (k-1)+\frac{t\lambda_1}{y}(k-y-1) }{\lambda_1} . 
	\end{equation}
	Then, by Definition~\ref{def} and \eqref{eqq11} we have that 
	\begin{equation}\label{eqq22}
		\Delta_3(\mathcal{P})=(k-y-1)(k-1)-(y-1)\frac{\left( r-\frac{t\lambda_1}{y}\right) (k-1)+\frac{t\lambda_1}{y}(k-y-1) }{\lambda_1}. 
	\end{equation}
	By Theorems~\ref{Ek} and ~\ref{Ekk}, we have that $\G$ is $2$-$\mathcal{P}$-homogeneous if and only if $\G$ is almost $2$-$\mathcal{P}$-homogeneous and $\Delta_3(\mathcal{P})=0$. From \eqref{eqq22}, it is easy to see that $\Delta_3(\mathcal{P})=0$ if and only if 
	\begin{eqnarray}\label{eqq33}
	r(k-1)-t\lambda_1=\frac{\lambda_1(k-y-1)(k-1)}{y-1}.
	\end{eqnarray}
	Therefore, from \eqref{eqq33} it is easy to see that $$t=\frac{(k-1)\left[r(y-1)-\lambda_1(k-y-1) \right] }{\lambda_1(y-1)}.$$
	The claim now immediately follows from the above comments and Proposition~\ref{k3}. 
\end{proof}

As we delve into our research, propositions \ref{k3} and \ref{k4} have surfaced, revealing a compelling problem that demands further investigation.

\begin{problem}
	Determine the existence  of a $\mathcal{D}=(\mathcal{P}, \mathcal{B}, \mathcal{I})$  quasi-symmetric $(v,b,r,k,\lambda_1,0)$ SPBIBD of type $(k-1,t)$  with intersection numbers $x=0$ and $y > 1$ whose incidence graph is (almost) $2$-$\mathcal{P}$-homogeneous. 
\end{problem}

\subsection{The (almost) $2$-$\mathcal{B}$-homogeneous condition}

This section investigates the (almost) $2$-$\mathcal{B}$-homogeneous properties of the incidence graph of certain SPBIBDs. Based on Theorem~\ref{he}, this incidence graph is a $(\mathcal{P}, \mathcal{B})$-bipartite distance-regularized graph, with all vertices having  eccentricity equal to four.

\begin{proposition}\label{r2}
	Let $\G$ denote the incidence graph of a $\mathcal{D}=(\mathcal{P}, \mathcal{B}, \mathcal{I})$  quasi-symmetric $(v,b,r,k,\lambda_1,0)$ SPBIBD of type $(k-1,t)$  with intersection numbers $x=0$ and $y=1$. The following \rm{(i)},\rm{(ii)} are equivalent:
	\begin{enumerate}[label=(\roman*),font=\rm]
		\item Each point appears in exactly $2$ blocks.
		\item $\G$ is isomorphic to the subdivision graph of a complete bipartite graph $K_{k,k}$ and $b=2k$.
	\end{enumerate}
	In this case, $\G$ is $2$-$\mathcal{B}$-homogeneous. 
\end{proposition}

\begin{proof}
	Similar to the proof of Proposition~\ref{k2}. 
\end{proof}

\begin{proposition}\label{B1}
	Let $\G$ denote the incidence graph of a $\mathcal{D}=(\mathcal{P}, \mathcal{B}, \mathcal{I})$  quasi-symmetric $(v,b,r,k,\lambda_1,0)$ SPBIBD of type $(k-1,t)$  with intersection numbers $x=0$ and $y=1$. Assume that $r\geq 3$. Then the following hold:
	\begin{enumerate}[label=(\roman*),font=\rm]
		\item $\G$ is almost $2$-$\mathcal{B}$-homogeneous if and only if $\mathcal{D}$ is a generalized quadrangle.
		\item $\G$ is not $2$-$\mathcal{B}$-homogeneous.
	\end{enumerate}
\end{proposition}

\begin{proof}
	By Theorem~\ref{he}, we observe that $\G$ is a $(\mathcal{P}, \mathcal{B})$-bipartite distance-regularized graph where every point and every block has eccentricity equal to four. Moreover, the intersection arrays of $\G$ can be computed in term of the parameters of $\mathcal{D}$. Therefore, to analyze the (almost) $2$-$\mathcal{B}$-homogeneous of $\G$, it suffices to compute the scalars $\Delta_2(\mathcal{B})$ and $\Delta_3(\mathcal{B})$, as defined in Definition~\ref{def}. For $(0\leq i \leq 4)$, let $c_i, b_i$ and $c'_i, b'_i$ denote the intersection numbers of the points and the blocks, respectively, as shown in Theorem~\ref{he}. Since $c_2'=y=1$, it follows that $c_2=\lambda_1=1$. We thus have that
	\begin{eqnarray}
		\Delta_2(\mathcal{B})&=&(b'_1-1)(c'_3-1)=(r-2)(t-1) \label{delta2b}, \\ \Delta_3(\mathcal{B})&=&(b_2-1)(c_4-1)=(r-2)(r-1) \label{delta3b}.
	\end{eqnarray}   
	By Theorem~\ref{Ekk}, we have that $\G$ is almost $2$-$\mathcal{B}$-homogeneous if and only if $\Delta_2(\mathcal{B})=0$. If $\mathcal{D}$ is a generalized quadrangle then $t=1$ which shows that $\Delta_2(\mathcal{B})=0$. Conversely, if $\Delta_2(\mathcal{B})=0$ then $t=1$, since $r>2$. This means that  $\mathcal{D}$ is a generalized quadrangle if and only if $\Delta_2(\mathcal{B})=0$. So, (i) follows. To prove the second part of our claim, we observe from \eqref{delta3b} that the scalar $\Delta_3(\mathcal{B})>0$ and so, by Theorem~\ref{Ek} we have that $\G$ is not $2$-$\mathcal{B}$-homogeneous.  This finishes the proof. 
\end{proof}

\begin{proposition}\label{k30}
	Let $\G$ denote the incidence graph of a $\mathcal{D}=(\mathcal{P}, \mathcal{B}, \mathcal{I})$  quasi-symmetric $(v,b,r,k,\lambda_1,0)$ SPBIBD of type $(k-1,t)$  with intersection numbers $x=0$ and $y>1$. In this case, $\G$ is almost $2$-$\mathcal{B}$-homogeneous if and only if $$r=\frac{(k-y)(\frac{t\lambda_1}{y}-1)(\lambda_1-1)}{\lambda_1(t-y)}+2.$$
\end{proposition}

\begin{proof}
	By Theorem~\ref{he}, we observe that $\G$ is a $(\mathcal{P}, \mathcal{B})$-bipartite distance-regularized graph where every point and every block has eccentricity equal to four. Moreover, the intersection arrays of $\G$ can be computed in term of the parameters of $\mathcal{D}$. Therefore, to analyze the (almost) $2$-$\mathcal{B}$-homogeneous condition of $\G$, it suffices to compute the scalars $p^2_{2,2}(\mathcal{B})$, as shown in Lemma~\ref{pi2iY}, and $\Delta_2(\mathcal{B})$, as defined in Definition~\ref{def}. For $(0\leq i \leq 4)$, let $c_i, b_i$ and $c'_i, b'_i$ denote the intersection numbers of the points and the blocks, respectively, as shown in Theorem~\ref{he}. Therefore, by Lemma~\ref{pi2iY} it follows that 
	\begin{equation}\label{eqqq1}
		p^2_{2,2}(\mathcal{B})=\frac{(k-y)\left( \frac{t\lambda_1}{y}-1\right) +y(r-2)}{y}. 
	\end{equation}
	Then, by Definition~\ref{def} and \eqref{eqq1} we have that 
	\begin{equation}\label{eqqq2}
		\Delta_2(\mathcal{B})=(r-2)\left(\frac{t\lambda_1}{y}-1 \right) -(\lambda_1-1)\frac{(k-y)\left( \frac{t\lambda_1}{y}-1\right) +y(r-2)}{y}. 
	\end{equation}
	By Theorem~\ref{Ekk}, we have that $\G$ is almost $2$-$\mathcal{B}$-homogeneous if and only if $\Delta_2(\mathcal{B})=0$.
	From \eqref{eqqq2}, it is easy to see that $\Delta_2(\mathcal{B})=0$ if and only if 
	\begin{eqnarray}\label{eqqq3}
		\lambda_1(r-2)(t-y)=(k-y)\left( \frac{t\lambda_1}{y}-1\right)(\lambda_1-1).
	\end{eqnarray}
	Since $y>1$, by Lemma \ref{cases} we also observe that $t>y$. Therefore, from \eqref{eqqq3} it is easy to see that $$r-2=\frac{(k-y)(\frac{t\lambda_1}{y}-1)(\lambda_1-1)}{\lambda_1(t-y)}.$$
	The claim now immediately follows from Theorem~\ref{Ekk}. 
\end{proof}

\begin{proposition}\label{k40}
	Let $\G$ denote the incidence graph of a $\mathcal{D}=(\mathcal{P}, \mathcal{B}, \mathcal{I})$  quasi-symmetric $(v,b,r,k,\lambda_1,0)$ SPBIBD of type $(k-1,t)$  with intersection numbers $x=0$ and $y>1$. In this case, $\G$ is $2$-$\mathcal{B}$-homogeneous if and only if $$r=\frac{(k-y)(\frac{t\lambda_1}{y}-1)(\lambda_1-1)}{\lambda_1(t-y)}+2, \hspace{0.5cm} t=\frac{(r-1)\left[k(\lambda_1-1)-y(r-\lambda_1-1) \right] }{\lambda_1(\lambda_1-1)}.$$
\end{proposition}

\begin{proof}
	By Theorem~\ref{he}, we observe that $\G$ is a $(\mathcal{P}, \mathcal{B})$-bipartite distance-regularized graph where every point and every block has eccentricity equal to four. Moreover, the intersection arrays of $\G$ can be computed in term of the parameters of $\mathcal{D}$. Therefore, to analyze the $2$-$\mathcal{B}$-homogeneous condition of $\G$, it suffices to compute the scalars $p^3_{2,3}(\mathcal{B})$, as shown in Lemma~\ref{pi2iY}, and $\Delta_3(\mathcal{B})$, as defined in Definition~\ref{def}. For $(0\leq i \leq 4)$, let $c_i, b_i$ and $c'_i, b'_i$ denote the intersection numbers of the points and the blocks, respectively, as shown in Theorem~\ref{he}. Therefore, by Lemma~\ref{pi2iY} it follows that 
	\begin{equation}\label{eqqqq11}
		p^3_{2,3}(\mathcal{B})= \frac{(k-t)(r-1)+t(r-\lambda_1-1) }{y} . 
	\end{equation}
	Then, by Definition~\ref{def} and \eqref{eqqqq11} we have that 
	\begin{equation}\label{eqqqq22}
		\Delta_3(\mathcal{B})=(r-\lambda_1-1)(r-1)-(\lambda_1-1)\frac{(k-t)(r-1)+t(r-\lambda_1-1) }{y}. 
	\end{equation}
	By Theorems~\ref{Ek} and ~\ref{Ekk}, we have that $\G$ is $2$-$\mathcal{B}$-homogeneous if and only if $\G$ is almost $2$-$\mathcal{B}$-homogeneous and $\Delta_3(\mathcal{B})=0$. From \eqref{eqqqq22}, it is easy to see that $\Delta_3(\mathcal{B})=0$ if and only if 
	\begin{eqnarray}\label{eqqqq33}
		\left[ (k-t)(r-1)+t(r-\lambda_1-1)\right] (\lambda_1-1)=y(r-1)(r-\lambda_1-1).
	\end{eqnarray}
	Therefore, from \eqref{eqqqq33} it is easy to see that $$t=\frac{(r-1)\left[k(\lambda_1-1)-y(r-\lambda_1-1) \right] }{\lambda_1(\lambda_1-1)}.$$
	The claim now immediately follows from the above comments and Proposition~\ref{k30}. 
\end{proof}

In the course of our research, Propositions \ref{k30} and \ref{k40} have brought forth a natural problem that warrants further investigation.

\begin{problem}
	Determine the existence  of a $\mathcal{D}=(\mathcal{P}, \mathcal{B}, \mathcal{I})$  quasi-symmetric $(v,b,r,k,\lambda_1,0)$ SPBIBD of type $(k-1,t)$  with intersection numbers $x=0$ and $y > 1$ whose incidence graph is (almost) $2$-$\mathcal{B}$-homogeneous. 
\end{problem}


\section{Distance-semiregular graphs and SPBIBDs} \label{sec:dsg}

\begin{lemma}\label{lemma20}
	Let $\G$ be a $(Y,Y')$-distance semiregular graph with respect to $Y$. Assume every vertex in $Y$ has eccentricity $D=4$. Let $b_i, c_i \; (0\leq i \leq 4)$ denote the intersection numbers of every vertex in $Y$. Then, $\G$ is the incidence graph of 
	a $(1+\frac{b_0b_1}{c_2}+\frac{b_0b_1b_2b_3}{c_2c_3c_4},b_0+\frac{b_0b_1b_2}{c_2c_3},b_0',b_0, c_2,0)$ SPBIBD of type $(b_1,c_3)$. 
\end{lemma}

\begin{proof}
	Since $\G$ is a $(Y,Y')$-distance semiregular graph with respect to $Y$, it is $(b_0,b_0')$-biregular where $b_0'$ denotes the valency of a vertex in $Y'$. We next consider the combinatorial incidence structure $\mathcal{D}=(Y, Y', \mathcal{I})$, with point set $Y$, block set $Y'$ and incidence $\mathcal{I}$. We will prove that $\mathcal{D}$ is a SPBIBD. As $\G$ is $(b_0,b_0')$-biregular, it follows every block has size $b_0'$ and every point is contained in $b_0$ blocks. Moreover, since $\G$ is bipartite and every point has eccentricity $4$, for every $p\in Y$ we have that $|Y|=|\G_0(p)|+|\G_2(p)|+|\G_4(p)|$ and $|Y'|=|\G_1(p)|+|\G_3(p)|$. The distance-regularity property around $p$ gives us that 
	\begin{eqnarray}
		|Y|&=&1+\frac{b_0b_1}{c_2}+\frac{b_0b_1b_2b_3}{c_2c_3c_4}, \hspace{1cm} |Y'|=b_0+\frac{b_0b_1b_2}{c_2c_3}.\nonumber 
	\end{eqnarray} 
	
 We next claim that every two points are contained either in $\lambda_1=c_2$ blocks or in $\lambda_2=0$ blocks. To prove our claim, pick $p\in Y$. Note that $p$ has eccentricity $4$ and since $\G$ is bipartite, any other point different from $p$ is at distance $2$ or at distance $4$ from $p$. Let $q\in \G_2(p)$. Then, we need to count the number of common neighbours of $p$ and $q$. We observe that $|\G(p)\cap \G(q)|=c_2$. Now, for $q\in \G_4(p)$ we observe that there is no block containing both $p$ and $q$. This proves our claim.
 
 Next, pick $p\in Y$ and $B\in Y'$. If $p$ is a neighbour of $B$ we observe that  $|\G(B)\cap \G_2(p)|=b_1$. So, if $(p,B)$ is a flag, then the number of points in $B$ which occur with $p$ in $c_2$ blocks is $b_1$. Similarly, if $(p,B)$ is a non-flag, then $B\in \G_3(p)$ and the number of points in $B$ which occur with $p$ in $c_2$ blocks is $|\G(B)\cap \G_2(p)|=c_3$. Therefore,  $\mathcal{D}$ is a SPBIBD of type $(b_1,c_3)$. 
\end{proof}

\begin{theorem}
	There is a one-to-one correspondence between the incidence graph of  SPBIBDs with parameters $(v,b,r,k,\lambda_1,0)$ of type $(k-1,t)$ where $0< t<k$  and distance-semiregular graphs with distance-regularized vertices of eccentricity $4$. 
\end{theorem}

\begin{proof}
By Lemma~\ref{lemma20}, $\G$ is the incidence graph of 
	a $(1+\frac{b_0b_1}{c_2}+\frac{b_0b_1b_2b_3}{c_2c_3c_4},b_0+\frac{b_0b_1b_2}{c_2c_3},b_0',b_0, c_2,0)$ SPBIBD of type $(b_1,c_3)$. Notice that $0<c_3<b_0'$ and $b_0<b_0+\frac{b_0b_1b_2}{c_2c_3}$. The result now immediately follows from Lemma~\ref{lemma11}.  
\end{proof}

\begin{lemma}\label{lemma2bis}
	Let $\G$ be a $(Y,Y')$-bipartite distance-regularized graph with vertices of eccentricity $4$. Let $b_i, c_i; b_i', c_i' \; (0\leq i \leq 4)$ denote the intersection numbers of every vertex in $Y$ and in $Y'$ respectively. Then, $\G$ is the incidence graph of 
	a $(1+\frac{b_0b_1}{c_2}+\frac{b_0b_1b_2b_3}{c_2c_3c_4},b_0+\frac{b_0b_1b_2}{c_2c_3},b_0',b_0, c_2,0)$ SPBIBD of type $(b_1,c_3)$ which is quasi-symmetric with intersection numbers $x=0$ and $y=c_2'$. 
\end{lemma}

\begin{proof}
	We next consider the combinatorial incidence structure $\mathcal{D}=(Y, Y', \mathcal{I})$, with point set $Y$, block set $Y'$ and incidence $\mathcal{I}$. Observe that every bipartite distance-regularized graph with vertices of eccentricity $4$ is either a bipartite distance-regular graph of diameter $4$ or a distance-biregular graph with $D=D'=4$. In particular, it is distance-semiregular with respect to both color partitions.  Therefore, since $\G$ has vertices of eccentricity $4$, by Lemma \ref{lemma20}, $\G$ is the incidence graph of a $(1+\frac{b_0b_1}{c_2}+\frac{b_0b_1b_2b_3}{c_2c_3c_4},b_0+\frac{b_0b_1b_2}{c_2c_3},b_0',b_0, c_2,0)$ SPBIBD of type $(b_1,c_3)$. We next assert that such a $1$-design is quasi-symmetric with $x=0$ and $y=c_2'$. To prove our claim, pick $B\in Y'$. We observe $\G_0(B)\cup\G_2(B)\cup\G_4(B)=Y'$ since vertices of $Y'$ have eccentricity $4$ and $\G$ is bipartite. Let $B' \in \G_2(B)\cup\G_4(B)$. Notice that both $\G_2(B)$ and $\G_4(B)$ are nonempty. If  $B' \in \G_4(B)$ then $B\cap B'=\emptyset$. Suppose now $ B' \in \G_2(B)$. Since $\G$ is distance-regularized we have $|\G(B)\cap \G(B')|=c_2'$. This shows that two intersecting blocks have the same number of points in common. Therefore, $\mathcal{D}$ is quasi-symmetric with $x=0$ and $y=c_2'$.
\end{proof}

\begin{theorem}\label{correspondence}
	There is a one-to-one correspondence between the incidence graph of  quasi-symmetric SPBIBDs with parameters $(v,b,r,k,\lambda_1,0)$ of type $(k-1,t)$ with intersection numbers $x=0$ and $y>0$, where $0< y\leq t<k$ , and bipartite distance-regularized graphs with vertices of eccentricity $4$. 
\end{theorem}

\begin{proof}
	Let $\G$ be a $(Y,Y')$-bipartite distance-regularized graphs with vertices of eccentricity $4$. Let $b_i, c_i; b_i', c_i' \; (0\leq i \leq 4)$ denote the intersection numbers of every vertex in $Y$ and in $Y'$ respectively. Then, by Lemma~\ref{lemma2bis}, $\G$ is the incidence graph of 
	a $(1+\frac{b_0b_1}{c_2}+\frac{b_0b_1b_2b_3}{c_2c_3c_4},b_0+\frac{b_0b_1b_2}{c_2c_3},b_0',b_0, c_2,0)$ SPBIBD of type $(b_1,c_3)$ which is quasi-symmetric with intersection numbers $x=0$ and $y=c_2'$.
	Notice that $0<c_2'\leq c_3<b_0'$ and $b_0<b_0+\frac{b_0b_1b_2}{c_2c_3}$. The result now immediately follows from Lemma~\ref{he}.  
\end{proof}

\begin{theorem}
	Let $\G$ be a $(Y,Y')$-bipartite distance-regularized graph. The following \rm{(i)},\rm{(ii)} are equivalent:
	\begin{enumerate}[label=(\roman*),font=\rm]
		\item $\G$ is $2$-$Y$-homogeneous with $D=4$ and $c'_2=1$.
		\item $\G$ is isomorphic to the subdivision graph of a complete bipartite graph $K_{n,n}=(X,\mathcal{R})$ with $n \geq 2$ and $Y=X$.
	\end{enumerate}
\end{theorem}

\begin{proof}
	Suppose that $\G$ is $2$-$Y$-homogeneous with $D=4$ and $c'_2=1$. We notice that every vertex in $Y'$ has eccentricity $4$. In fact, if $\G$ is regular then $\G$ is distance-regular and so, every vertex has the same eccentricity. Otherwise, by \cite[Proposition 5.8]{FR} we have that $D=D'=4$. By Theorem~\ref{correspondence}, we therefore have that $\G$ is the incidence graph of a quasi-symmetric $(v,b,r,k,\lambda_1,0)$ SPBIBD of type $(k-1,t)$ with intersection numbers $x=0$ and $y=1$. We observe that either $k\leq 2$ or $r\leq 2$ as otherwise, by Propositions~\ref{P1} and \ref{B1}, we have that $k\geq 3$ and $r\geq 3$ contradicts the fact that  $\G$ is $2$-$Y$-homogeneous. Moreover, since $\G$ is connected and $D=D'=4$, either $k=2$ or $r=2$. Then, by Propositions~\ref{k2} and ~\ref{r2} we have that $\G$ is isomorphic to the subdivision graph of a complete bipartite graph $K_{n,n}=(X,\mathcal{R})$ with $n \geq 2$ and $Y=X$. Conversely, if $\G$ is isomorphic to the subdivision graph of a complete bipartite graph $K_{n,n}=(X,\mathcal{R})$ with $n \geq 2$ and $Y=X$, then $\G$ is a $(Y,Y')$-bipartite distance-regularized graph with vertices of valency $2$, $D=4$ and $c'_2=1$; see for instace \cite[Theorem 2.6]{FBPS}. If $\G$ is regular then it is isomorphic to a cycle of lenth $8$, which is $2$-$Y$-homogeneous with $D=4$ and $c'_2=1$. Otherwise, $\G$ is distance-biregular with $k'=2$, which is $2$-$Y$-homogeneous by \cite[Theorem 4.2]{FBPS}.
\end{proof}




{\small
	\bibliographystyle{references}
	\bibliography{almost2homogDBRG}
}

\end{document}